\documentclass[10pt, a4paper, reqno, oneside]{amsart}

\usepackage{amsmath, amsfonts, amsthm, enumerate, multicol, scalefnt, relsize}

\usepackage[mathscr]{euscript}
\usepackage{mathbbol}
\usepackage[latin1]{inputenc}
\usepackage{graphicx}
\usepackage[all]{xy}
\usepackage{enumitem}
\usepackage{autobreak}
\setlist[itemize]{noitemsep, topsep=1pt, leftmargin=20pt}
\usepackage{xfrac}
\usepackage{todonotes}

\usepackage{fullpage}

\usepackage[hypertexnames=false,
backref=page,
    pdfpagemode=UseNone,
    breaklinks=true,
    extension=pdf,
    colorlinks=true,
    linkcolor=blue,
    citecolor=blue,
    urlcolor=blue,
]{hyperref}

\newcommand\bcdot{\ensuremath{
  \mathchoice
   {\mskip\thinmuskip\lower0.2ex\hbox{\scalebox{1.6}{$\cdot$}}\mskip\thinmuskip}}
   {\mskip\thinmuskip\lower0.2ex\hbox{\scalebox{1.6}{$\cdot$}}\mskip\thinmuskip}
   {\lower0.3ex\hbox{\scalebox{1.2}{$\cdot$}}}
   {\lower0.3ex\hbox{\scalebox{1.2}{$\cdot$}}}
}
\theoremstyle{plain}
\newtheorem{theo}{Theorem}[section]

\newtheorem{prop}[theo]{Proposition}

\theoremstyle{definition}
\newtheorem{rem}[theo]{Remark}

\theoremstyle{plain}
\newtheorem{lemma}[theo]{Lemma}
\newtheorem{theorem}[theo]{Theorem}

\theoremstyle{definition}

\newtheorem{remark}[theo]{Remark}

\theoremstyle{plain}
\newtheorem{thmint}{Theorem}

\renewcommand{\=}{:=}

\renewcommand{\a}{\alpha}
\renewcommand{\b}{\beta}

\renewcommand{\d}{\delta}
\newcommand{\e}{\varepsilon}

\newcommand{\g}{\gamma}
\newcommand{\h}{\eta}

\renewcommand{\l}{\lambda}
\newcommand{\w}{\omega}
\newcommand{\q}{\vartheta}
\newcommand{\s}{\sigma}

\newcommand{\D}{\Delta}

\newcommand{\G}{\Gamma}

\renewcommand{\L}{\Lambda}
\renewcommand{\S}{\Sigma}

\newcommand{\W}{\Omega}

\DeclareSymbolFontAlphabet{\mathbb}{AMSb}
\DeclareSymbolFontAlphabet{\mathbbl}{bbold}

\newcommand{\bC}{\mathbb{C}}

\newcommand{\bR}{\mathbb{R}}

\newcommand{\bN}{\mathbb{N}}

\newcommand\Sym{\mathrm{Sym}}
\newcommand\Skew{\mathrm{Skew}}

\newcommand{\cC}{\mathcal{C}}

\newcommand{\eL}{\EuScript{L}}
\newcommand{\eM}{\EuScript{M}}

\newcommand{\eT}{\EuScript{T}}
\newcommand{\eU}{\EuScript{U}}

\DeclareMathOperator\tr{tr}
\DeclareMathOperator\Tr{Tr}

\DeclareMathOperator\End{End}
\DeclareMathOperator\Hess{Hess}
\DeclareMathOperator\scal{scal}

\DeclareMathOperator\Id{Id}
\DeclareMathOperator\diff{d\!}

\newcommand{\p}{\partial}
\newcommand{\ti}{\mathtt{i}}
\newcommand{\Ric}{\operatorname{Ric}}

\newcommand{\ol}{\overline}
\newcommand{\zero}{\operatorname{o}}
\newcommand{\im}{\operatorname{Im}}
\newcommand{\n}{\nabla}

\makeatletter
\@namedef{subjclassname@2020}{\textup{2020} Mathematics Subject Classification}
\makeatother

\allowdisplaybreaks[3]

\title[]{On the linearization stability of\\
the Chern-scalar curvature}

\author{Daniele Angella}
\address[Daniele Angella]{Dipartimento di Matematica e Informatica ``Ulisse Dini''\\
Universit\`a degli Studi di Firenze,
viale Morgagni 67/a,
	50134 Firenze, Italy}
\email{daniele.angella@unifi.it}
\email{daniele.angella@gmail.com}

\author{Francesco Pediconi}
\address[Francesco Pediconi]{Dipartimento di Matematica e Informatica ``Ulisse Dini''\\
	Universit\`a degli Studi di Firenze,
	viale Morgagni 67/a,
	50134 Firenze, Italy}
\email{francesco.pediconi@unifi.it}

\thanks{
The authors are supported by project PRIN2017 ``Real and Complex Manifolds: Topology, Geometry and holomorphic dynamics'' (code 2017JZ2SW5), and by GNSAGA of INdAM
} 

\begin{document}

\begin{abstract}
In this note, we study the local properties of the Chern-scalar curvature function by looking at its linearization. In particular, we study its linearization stability and the structure of the space of Hermitian metrics with prescribed Chern-scalar curvature.
\end{abstract}

\maketitle

\section{Introduction}

In Riemannian geometry, the space of Riemannian metrics $\eM$ attached to a differentiable manifold $M$ plays a crucial role as a differentiable invariant. From this point of view, the function $\scal$, assigning to any Riemannian metric $g$ its scalar curvature $\scal(g)$, is well understood. We refer e.g. to the seminal works \cite{kazdan-warner-AnnMath-1, kazdan-warner-JDG, kazdan-warner-AnnMath-2, kazdan-warner-InvMath, berardbergery}. For compact surfaces $S$, the image of this function depends on the Euler characteristic. Indeed, the necessary and sufficient conditions for a smooth function to be the Gaussian curvature of some metric are either to be positive somewhere when $\chi(S)>0$, or to change sign or to be identically zero when $\chi(S)=0$, or to be negative somewhere when $\chi(S)<0$, see \cite[Thm 6.3, Thm 11.8]{kazdan-warner-AnnMath-1}, \cite[Thm 5.6]{kazdan-warner-AnnMath-2}.
There is a similar trichotomy in higher dimension, where one makes advantage of the sign of the first eigenvalue of the conformal Laplacian operator, see \cite[Thm 6.4]{kazdan-warner-AnnMath-2}.
Two main ingredients in the work by Kazdan and Warner are the study of the local surjectivity of the map $g\mapsto \scal(g)$ by means of the Inverse Function Theorem \cite[Lemma 2, p 228]{kazdan-warner-InvMath}, and an Approximation Lemma for studying the $L^p$-closure of the orbits of a function under the action of the diffeomorphism group of $M$ \cite[p. 228]{kazdan-warner-InvMath}.
This note is born as an attempt to understand similar questions concerning the geometry of Hermitian metrics on a compact complex manifold.

In Hermitian geometry, the Levi-Civita connection is replaced by Hermitian connections with possibly non-zero torsion. Among Hermitian connections, there are some canonical choices (in the sense of \cite{gauduchon-bumi}), in particular, the {\em Chern connection} is uniquely characterized by having torsion of type $(2,0)$, equivalenty, by the $(0,1)$-component coinciding with the Cauchy-Riemann operator of the holomorphic tangent bundle, see e.g. \cite[p. 273]{gauduchon-bumi}.
The {\em Chern-scalar curvature} $\scal^{\rm Ch}$, obtained by tracing the curvature of such connection (see Section \ref{prel} for the definition), has been investigated by several authors. Hermitian metrics on compact complex manifolds with constant Chern-scalar curvature have been investigated and constructed in \cite{angella-calamai-spotti, koca-lejmi, angella-pediconi} and others. A first result on prescribing the Chern-scalar curvature appear in \cite[Thm 1.1]{ho}, and the problem has been recently addressed in \cite{fusi} by the conformal methods of Kazdan and Warner.

\smallskip

In this note, we exploit the techniques by \cite{fischer-marsden, kazdan-warner-InvMath, bourguignon, berardbergery} to study the relationship between infinitesimal and actual deformations of the Chern-scalar curvature function with respect to a varying metric.
Concerning the analogue results for the Riemannian scalar curvature, we refer in particular to \cite[Thm A, Thm A', Thm 7.9]{fischer-marsden}.

Let $(M,J)$ be a compact, connected, complex manifold of dimension $\dim_{\bR}M=2m$, and $\eM_{\rm H}$ the space of smooth Hermitian metrics on it.
We study the local properties of the function $\scal^{\rm Ch} : \eM_{\rm H} \to \cC^{\infty}(M,\bR)$ by looking at its linearization (see Proposition \ref{prop:first-variation-scal} and Proposition \ref{prop:second-variation-scal}).
Let us recall that $\scal^{\rm Ch}$ is said to be {\em linearization stable at a metric $g_{\zero}\in\eM_{\rm H}$} if, for any direction $h \in \ker(\scal^{\rm Ch})'_{g_{\zero}}$, there exists a smooth path $g : (-\epsilon,\epsilon) \to \eM_{\rm H}$ such that $g(0)=g_{\zero}$, $\dot{g}(0)=h$ and $\scal^{\rm Ch}(g(t))=\scal^{\rm Ch}(g_{\zero})$ for any $-\epsilon<t<\epsilon$. On the other hand, if $\scal^{Ch}$ is not linearization stable at $g_{\zero}$, it is said to be {\em linearization unstable at $g_{\zero}$}.
Our first result reads as follows:

\begin{thmint}\label{thm:a}
Let $g \in \eM_{\rm H}$, set $\l \= \scal^{\rm Ch}(g) \in \cC^{\infty}(M,\bR)$ and assume that either $g$ is not first-Chern-Einstein, or $g$ is first-Chern-Einstein with $\tfrac{\l}m<\diff^{\,*}\!\q$.
Then, the function $\scal^{\rm Ch} : \eM_{\rm H} \to \cC^{\infty}(M,\bR)$ is linearization stable at $g$ and maps any neighborhood of $g$ onto a neighborhood of $\l$.
\end{thmint}

In the statement of the theorem above, $\q$ denotes the torsion $1$-form of $g$ (see Equation \eqref{eq:def-lee}). Moreover, the {\em first-Chern-Einstein condition} is a generalization of the K\"ahler-Einstein equation in the Hermitian, possibly non-K\"ahlerian, setting (see Section \ref{prel}).
In particular, since $\q=0$ and $\scal^{\rm Ch}=\scal$ at any K\"ahler metric, we remark that this theorem applies to K\"ahler-Einstein metrics with negative scalar curvature.

Our second result concerns the structure of the space
$\eM_{\rm H}(\l) \= \{g \in \eM_{\rm H} : \scal^{\rm Ch}(g)=\l\}$
of Hermitian metrics with prescribed Chern-scalar curvature. More precisely, we prove

\begin{thmint}\label{thm:b}
Let $\l \in \cC^{\infty}(M,\bR)$ and assume that $\eM_{\rm H}(\l)$ is not empty. If either $(M,J)$ is non-K\"ahlerian and $c_1^{\mathsmaller{\rm BC}}(M,J) \neq 0$, or $(M,J)$ is K\"ahlerian and $c_1(M,J)$ has no sign, then $\eM_{\rm H}(\l)$ is a closed smooth ILH-submanifold of $\eM_{\rm H}$ with tangent space $T_g\eM_{\rm H}(\l) = \ker(\scal^{\rm Ch})'_g$ at $g \in \eM_{\rm H}$.
\end{thmint}

Here, we say that $(M,J)$ is K\"ahlerian if it admits K\"ahler metrics, non-K\"ahlerian otherwise. We also denoted by $c_1^{\mathsmaller{\rm BC}}(M,J)$ and $c_1(M,J)$ the first Chern class of $(M,J)$ in the Bott-Chern and de Rham cohomology, respectively. Moreover, for the notion of infinite-dimensional ILH-manifold, we refer to \cite[Ch II]{omori}.
Notice that both the hypotheses stated in Theorem \ref{thm:b} assure that the manifold $(M,J)$ does not admit any first-Chern-Einstein metric, which is a key point in the proof.
Let us stress also that, while in the Riemannian case the prescribed scalar curvature problem is well understood (see \cite[Thm 4.35]{besse}), the question whether $\eM_{\rm H}(\l)$ is non-empty is far from being answered. Therefore, we collect in Remark \ref{rmk:non-empty} the state of the art, up to our knowledge.

Finally, in our last result, we exhibit concrete examples of Hermitian metrics at which $\scal^{\rm Ch}$ is linearization unstable. Notice that, by Theorem \ref{thm:a}, any such metric $g$ is necessarily first-Chern-Einstein with $\frac{1}{m}\scal^{\rm Ch}(g)-\diff^{\,*}\!\q\geq0$ at some point. More precisely, we prove

\begin{thmint}\label{thm:c}
Let $g_{\zero}$ be a K\"ahler-Einstein metric with positive scalar curvature. If $(M,g_{\zero})$ admits global non-trivial Killing vector fields, then $\scal^{\rm Ch}$ is linearization unstable at $g_{\zero}$.
\end{thmint}

In particular, this theorem applies to the K\"ahler-Einstein metrics on compact Fano manifolds admitting an isometric Lie group action.
On the other hand, we do not have any information on the general class of first-Chern-Einstein metrics with $\frac{1}{m}\scal^{\rm Ch}-\diff^{\,*}\!\q\geq0$ at some point, which includes Ricci-flat K\"ahler metrics and K\"ahler-Einstein metrics with positive scalar curvature without global non-trivial Killing vector fields.

\medskip

The paper is organized as follows.
In Section \ref{prel}, we summarize some basic facts on Geometric Analysis, on Complex Linear Algebra and on the Chern connection. In Section \ref{sec:variation}, we compute the first and second variation formulas for the Chern-scalar curvature, proving Proposition \ref{prop:first-variation-scal} and Proposition \ref{prop:second-variation-scal}.
In Section \ref{sec:main}, we prove the main results Theorem \ref{thm:a}, Theorem \ref{thm:b} and Theorem \ref{thm:c}.

\bigskip

\noindent {\itshape Acknowledgements.} The authors are warmly grateful to Matteo Focardi and Fabio Podest\`a for useful discussions.

\section{Preliminaries and notation} \label{prel}
\setcounter{equation} 0

In this section, we summarize some basic facts on geometric analysis on Riemannian manifolds, referring to e.g. \cite{palais, Hebey, aubin, besse}, and we set the notation and some preliminary results concerning the geometry of Hermitian metrics, see e.g. \cite{gauduchon-bumi}.

\subsection{Basics on geometric analysis}

\subsubsection{A consequence of the Implicit Function Theorem} \hfill \par

Let $X,Y$ be Banach spaces, $\eU \subset X$ an open set and $L(X,Y)$ the Banach space of continuous linear maps $T: X \to Y$. A map $f: \eU \to Y$ is said to be {\it of class $\cC^1$} if there exists a continuous map $\diff f: \eU \to L(X,Y)
$ called {\it differential of $f$} such that
$$
\diff f(x)(v) = \lim_{t \to 0} \tfrac1t(f(x+tv)-f(x)) \quad \text{ for any } x \in \eU \, , \,\, v \in X \,\, .
$$
Let us consider the Banach space $L^{(k)}(X,Y)$, $k \in \bN$, given by
$$\begin{gathered}
L^{(k)}(X,Y) \= \big\{ \text{$k$-multilinear continuous maps } T: \underbrace{X{\times}{\dots}{\times}X}_{\mathsmaller{k\text{-times}}} \to Y\big\} \,\, , \\
\|T\|_{L^{(k)}(X,Y)} \= \sup \big\{|T(x_1,{\dots},x_k)|_Y : x_1,{\dots},x_k \in X \, , \,\, |x_1|_X={\dots}=|x_k|_X=1 \big\} \,\, .
\end{gathered}$$
It can be directly checked that $L(X,L^{(k-1)}(X,Y)) \simeq L^{(k)}(X,Y)$ for any $k \in \bN$. Therefore, this allows to give the following recursive definition: a map $f: \eU \to Y$ is said to be {\it of class $\cC^k$} if it is of class $\cC^{k-1}$ and there exists a continuous map $\diff^{\,(k)}\!f: \eU \to L^{(k)}(X,Y)$ such that $\diff^{\,(k)}\!f(x) = \diff\big(\diff^{\,(k-1)}\!f\big)(x)$ for any $x \in \eU$. As usual, $f$ is said to be {\it smooth} if it is of class $\cC^k$ for any $k \in \bN$. For the sake of shortness, we set $f'_x(v) \= \diff f(x)(v)$ and $f''_x(v) \=\diff^{\,(2)}\!f(x)(v)$. \smallskip

Let now $f: \eU \to Y$ be a smooth map and $x_{\zero} \in \eU$ a point. We recall that $f$ is said to be a {\it submersion at $x_{\zero}$} if $f'_{x_{\zero}}$ is surjective and the exact short sequence $\{0\} \to \ker(f'_{x_{\zero}}) \to X \to Y \to \{0\}$ splits, i.e. there exists a closed subspace $Z_{x_{\zero}} \subset X$ such that $X = \ker(f'_{x_{\zero}}) \oplus Z_{x_{\zero}}$ and the restriction $f'_{x_{\zero}}|_{Z_{x_{\zero}}}$ is an isomorphism of Banach spaces from $Z_{x_{\zero}}$ to $Y$. By the Implicit Function Theorem, see e.g. \cite[p. 72]{aubin}, if $f$ is submersion at $x_{\zero}$, it follows that:
\begin{itemize}
\item[i)] $f$ is {\it locally surjective at $x_{\zero}$}, i.e. $f$ maps any neighborhood of $x_{\zero}$ in $X$ onto a neighborhood of $f(x_{\zero})$ in $Y$;
\item[ii)] the preimage $S_{f(x_{\zero})} \= f^{-1}(f(x_{\zero}))$ is a smooth submanifold of $X$ in a neighborhood of $x_{\zero}$ with tangent space $T_{x_{\zero}}S = \ker(f'_{x_{\zero}})$.
\end{itemize}
Finally, we recall that $f: \eU \to Y$ is said to be {\it linearization stable at $x_{\zero} $} if, for any $v \in \ker(f'_{x_{\zero}}) \subset X$, there exists a smooth path $x : (-\epsilon,\epsilon) \to X$ such that $x(0)=x_{\zero}$, $\dot{x}(0)=v$ and $f(x(t))=f(x_{\zero})$ for any $-\epsilon<t<\epsilon$, see \cite[p. 519]{fischer-marsden}. Notice that, if $f$ is submersion at $x_{\zero}$, then it is also linearization stable at $x_{\zero}$, but the converse assertion is not true.

\subsubsection{Sobolev spaces on Riemannian manifolds} \hfill \par

Let $M$ be a connected, compact, oriented smooth manifold without boundary of even dimension $\dim M = 2m$, $g$ a fixed background Riemannian metric on $M$ and $\nu_g$ the induced Riemannian volume form. We extend $g$ to a Riemannian metric on the fibers of the bundle $\eT^{(r,s)} M \to M$ of $(r,s)$-tensors over $M$ in the usual way and we denote by $D^{g} : \cC^{\infty}(M,\eT^{(r,s)} M) \to \cC^{\infty}(M,\eT^{(r,s+1)} M)$ the Levi-Civita connection of $g$. \smallskip

Let $E \to M$ be any vector subbundle of $\eT^{(\cdot,\cdot\cdot)}M$. For any integer $k \geq 0$ and for any tensor fields $h_1,h_2 \in \cC^{\infty}(M,E)$, we define the bilinear form
$$
\langle h_1, h_2 \rangle_{W^{k,2}} \= \sum_{0\leq i \leq k} \int_M g\big((D^{g})^i h_1,(D^{g})^i h_2\big) \,\nu_g \,\, .
$$
Since $M$ is compact, the topology induced by the norm $\| \cdot \|_{W^{k,2}} \= \sqrt{\langle \cdot , \cdot \rangle_{W^{k,2}}}$ is independent of the Riemannian metric $g$ (see e.g. \cite[Prop 2.2]{Hebey}). Accordingly, the Sobolev space $W^{k,2}(M,E)$ is defined as the completion of $\cC^{\infty}(M,E)$ with respect to the norm $\| \cdot \|_{W^{k,2}}$. For the sake of notation, we set $L^2 \= W^{0,2}$.

Notice that $\big(W^{k,2}(M,E), \langle {\phantom{i}},{\phantom{i}}\! \rangle_{W^{k,2}}\big)$ is a Hilbert space. Moreover, for any $k \geq m+1$, by the Sobolev Embeddings Theorem (see e.g. \cite[Thm 2.7]{Hebey}) there exists a continuous embedding
\begin{equation} \label{embSob}
\big(W^{k,2}(M,E),\, \|\cdot\|_{W^{k,2}}\big) \,\,\hookrightarrow\,\, \big(\cC^{k-m-1}(M,E),\, \|\cdot\|_{\cC^{k-m-1}}\big) \,\, ,
\end{equation}
where $\|\cdot\|_{\cC^{k'}}$ denotes the usual $\cC^{k'}$-norm for any integer $k'\geq0$.

\begin{remark}
We stress that \eqref{embSob} implies that $W^{k,2}(M,E)$ consists of continuous sections if $k \geq m+1$. In particular, it is possible to define $W^{k,2}(M,E)$ with $k \geq m+1$ for any subbundle $E$ of $\eT^{(\cdot,\cdot\cdot)}M$. A remarkable example is the {\it space of Riemannian metrics of class $W^{k,2}$} defined as
$$
\eM^k \=  W^{k,2}\big(M,S^2_+(T^*M)\big) \,\, .
$$
\end{remark}

Let now $E,F$ be two vector subbundle of $\eT^{(\cdot,\cdot\cdot)} M$ and $P: \cC^{\infty}(M,E) \to \cC^{\infty}(M,F)$ a linear differential operator of order $r$. We recall that: \begin{itemize}
\item[$\bcdot$] the {\it principal symbol of P} is the endomorphism $\s(P): T^*M \otimes E \to F$ defined as follows: for any $x \in M$, $\xi \in T_x^*M$, $s \in E_x$
\begin{equation} \label{symbol}
\s(P)_x(\xi\otimes s) \= \tfrac1{r!}P(\phi^ru)(x) \,\, ,
\end{equation}
where $\phi \in \cC^{\infty}(M,\bR)$ verifies $\phi(x)=0$, $\diff\phi_x=\xi$ and $u \in \cC^{\infty}(M,E)$ verifies  $u(x)=s$;
\item[$\bcdot$] the {\it $L^2$-adjoint of $P$} is the unique linear differential operator $P^*: \cC^{\infty}(M,F) \to \cC^{\infty}(M,E)$ of order $r$ satisfying
$$
\langle P(s_1), s_2\rangle_{L^2} = \langle s_1, P^*(s_2) \rangle_{L^2} \,\,\, \text{ for any } s_1 \in \cC^{\infty}(M,E) \, , \, s_2 \in \cC^{\infty}(M,F) \,\, .
$$
\end{itemize}
We also remark that $P$ can be uniquely extended to a linear differential operator $P: W^{k+r,2}(M,E) \to W^{k,2}(M,F)$ for any integer $k \geq 0$ (see \cite[Thm 6, p. 152]{palais}). Then, we have the following
 
\begin{theorem}[{Berger-Ebin Splitting Lemma \cite[Thm 4.1]{berger-ebin}}] \label{thm:BE}
Let $P: \cC^{\infty}(M,E) \to \cC^{\infty}(M,F)$ be a linear differential operator of order $r$ and $k \in \bN$ such that $k \geq r$. If $P$ has injective symbol or its $L^2$-adjoint $P^*$ has injective symbol, then
$$
W^{k,2}(M,F) = \im(P) \oplus \ker(P^*) \,\, ,
$$
where $P$ is extended to $P: W^{k+r,2}(M,E) \to W^{k,2}(M,F)$ and, consequently, $P^*: W^{k,2}(M,F) \to W^{k-r,2}(M,E)$. Moreover, if $P$ has injective symbol, then $\ker(P) \subset W^{k+r,2}(M,E)$ is finite dimensional and consists of smooth sections.
\end{theorem}

\begin{rem} \label{rem:adjP}
As a consequence of the Berger-Ebin Splitting Lemma, we remark that: {\it if $P^*$ is injective and has injective symbol, then $P$ is surjective and its kernel splits}.
\end{rem}

\subsection{Complex linear algebra} \hfill \par

Let $V=(V,J,g)$ be a triple given by a real vector space $V$ of dimension $\dim_{\bR}V=2m$, a linear complex structure $J$ on $V$ and an Euclidean scalar product $g$ on $V$ such that $g(J(\cdot),J(\cdot\cdot))=g(\cdot,\cdot\cdot)$. The complexification $V^{\bC}\=V \otimes_{\bR} \bC$ splits as a sum of $J$-eigenspaces $V^{\bC} = V^{1,0} \oplus V^{0,1}$ and all the real tensors on $V$ can be uniquely $\bC$-linearly extended to $V^{\bC}$. Fix a $(J,g)$-unitary basis $(e_i,Je_i)$ for $V$,
and consider the associated complex basis
$$
\e_i \= \tfrac1{\sqrt2}(e_i-\ti Je_i) \,\, , \quad \e_{\bar{i}} \= \tfrac1{\sqrt2}(e_i+\ti Je_i)
$$
for $V^{\bC}$, which is unitary with respect to the Hermitian extension of $g$ to $V^{\bC}$.
Clearly, it holds that $\overline{\e_i}=\e_{\bar{i}}$ and $J\e_i=\ti \e_i$, $J\e_{\bar{i}}=-\ti \e_{\bar{i}}$. Moreover, $J$ acts on covectors $\q \in V^*$ via $(J\q) \= - \q \circ J$, so that $(e^i,Je^i)$ is the dual basis of $(e_i,Je_i)$ for $V^*$. For the complexification, we get that
$$
\e^i \= \tfrac1{\sqrt2}(e^i+\ti Je^i) \,\, , \quad \e^{\bar{i}} \= \tfrac1{\sqrt2}(e_i-\ti Je_i)
$$
is the dual basis of $(\e_i,\e_{\bar{i}})$, and $J\e^i=-\ti \e^i$, $J\e^{\bar{i}}=\ti \e^{\bar{i}}$. With respect to such basis, we have
$$
g = \d_{\bar{j}i} \, \e^i \odot \e^{\bar{j}} \,\, , \quad \text{ with } \,\, \e^i \odot \e^{\bar{j}} \= \e^i \otimes \e^{\bar{j}} + \e^{\bar{j}} \otimes \e^i \,\, .
$$
We consider now the spaces
$$\begin{aligned}
\Sym^{1,1}(V) &\= \{h \in \End(V) : g(h(\cdot),\cdot\cdot)=g(\cdot,h(\cdot\cdot)) \, , \,\, [h,J]=0 \} \,\, , \\
\Skew^{1,1}(V) &\= \{\tilde{h} \in \End(V) : g(\tilde{h}(\cdot),\cdot\cdot)=-g(\cdot,\tilde{h}(\cdot\cdot)) \, , \,\, [\tilde{h},J]=0 \} \,\, .
\end{aligned}$$
Notice now that any $h \in \Sym^{1,1}(V)$ preserves the decomposition $V^{\bC} = V^{1,0} \oplus V^{0,1}$ and takes the form
$$
h = h^j_{{\phantom{j}}i} \, \e_j \otimes \e^i + h^{\bar{j}}_{{\phantom{\bar{j}}}\bar{i}} \, \e_{\bar{j}} \otimes \e^{\bar{i}}  \,\, , \quad \text{ with } \,\, h^j_{{\phantom{j}}i} \in \bR \,\, \text{ and } \,\, h^{\bar{j}}_{{\phantom{\bar{j}}}\bar{i}} = h^j_{{\phantom{j}}i} = h^i_{{\phantom{i}}j} \,\, .
$$
Moreover, the linear map
$$
\Sym^{1,1}(V) \to \Skew^{1,1}(V) \,\, , \quad h \mapsto \tilde{h} = J \circ h
$$
is an isomorphism, with inverse given by $\tilde{h} \mapsto h= -J \circ \tilde{h}$. Then, we denote by $\tr^{\bR}: \Sym^{1,1}(V) \to \bR$ the trace of the real endomorphism $h: V \to V$ and by $\tr^{\bC}: \Sym^{1,1}(V) \to \bR$ the trace of the complex endomorphism $h: V^{1,0} \to V^{1,0}$, which are related by
$$
\tr^{\bR}(h) = \sum_{1 \leq i \leq m} \big(g(h(e_i),e_i) + g(h(Je_i),Je_i)\big) = 2 \sum_{1 \leq i \leq m} g(h(\e_i),\e_{\bar{i}}) = 2\tr^{\bC}(h) \,\, .
$$
Finally, we consider the space
$$
\L^{1,1}(V^*) \= \{ \a \in \L^2(V^*) : a(J(\cdot),J(\cdot\cdot))=\a(\cdot,\cdot\cdot) \} \,\, ,
$$
and we observe that the linear map
$$
\rho_g: \Sym^{1,1}(V) \to \L^{1,1}(V^*) \,\, , \quad \rho_g(h) \= g((J \circ h) \,\cdot,\cdot\cdot)
$$
is an isomorphism. Accordingly, we define
$$
\Tr^{\bC}_g : \L^{1,1}(V^*) \to \bR \,\, , \quad \Tr^{\bC}_g(\a) \= \tr^{\bC}\big(\rho_g^{-1}(\a)\big) \,\, .
$$
Since any $\a \in \L^{1,1}(V^*)$ is of the form
$$
\a = \a_{\bar{j}i} \,\ti\, \e^i \wedge \e^{\bar{j}} \,\, , \quad \text{ with } \,\, \e^i \wedge \e^{\bar{j}} \= \e^i \otimes \e^{\bar{j}} - \e^{\bar{j}} \otimes \e^i \,\, ,
$$
an easy computation shows that
$$
\Tr^{\bC}_g(\a) = \sum_{1 \leq i \leq m} \a(e_i,Je_i) = -\ti\sum_{1 \leq i \leq m} \a(\e_i,\e_{\bar{i}}) = \d^{i\bar{j}}\a_{\bar{j}i} \,\, .
$$

\subsection{The Chern connection} \hfill \par

Let $(M,J,g)$ be a compact, connected, Hermitian manifold of dimension $\dim_{\bR}M=2m$ and let $\w \= \rho_g(\Id)=g(J\cdot, \cdot)$ be its fundamental $2$-form, where $\Id\in\cC^{\infty}(M,\Sym^{1,1}(TM))$ is the identity endomorphism, $D^g$ its Levi-Civita connection and $\n$ its Chern connection, defined by
\begin{equation} \label{defChern}
g(\n_XY,Z) \= g(D^g_XY,Z)-\tfrac12\diff\w(JX,Y,Z) 
\end{equation}
for any $X,Y,Z \in \cC^{\infty}(M,TM)$. It is well-known that the Chern connection is characte\-rized by the following properties, see e.g. \cite[p. 273]{gauduchon-bumi}:
$$
\n g = 0 \,\, , \quad \n J =0 \,\, , \quad J(T(X,Y))=T(JX,Y)=T(X,JY) \,\, , 
$$
where $T(X,Y) \= \n_XY -\n_YX - [X,Y]$ is the {\em torsion tensor of $\n$}, which can be expressed by (see e.g. \cite[Prop 4]{gauduchon-bumi})
$$
-2g(T(X,Y),Z) = \diff\w(JX,Y,Z) + \diff\w(X,JY,Z) \,\, .
$$
We denote by
\begin{equation}\label{eq:def-lee}
\q(X) \= \tr^{\bR}(T(X,\cdot)) = \Tr^{\bC}_g(X\lrcorner\diff\w)
\end{equation}
the {\it Lee form}. We also define the {\it Chern-curvature operator} by
\begin{equation} \label{Chern-curv}
\W(g) \in \cC^{\infty}(M,\L^{1,1}(T^*M) \otimes \Skew^{1,1}(TM)) \,\, , \quad \W(g)(X,Y) \= \n_{[X,Y]} - [\n_X,\n_Y] \,\, .
\end{equation}
Moreover, we call {\it first Chern-Ricci form} the tensor field
\begin{equation}\label{eq:def-tildeS}
\tilde{S}(g) \in \cC^{\infty}(M,\L^{1,1}(T^*M)) \,\, , \quad \tilde{S}(g)(X,Y) \= -\tr^{\bC}(J \circ \W(g)(X,Y))
\end{equation}
and {\it first Chern-Ricci symmetric endomorphism}
$$
S(g) \in \cC^{\infty}(M,\Sym^{1,1}(TM)) \,\, , \quad S(g) \= \rho_g^{-1}\big(\tilde{S}(g)\big) \,\, .
$$
Finally, the {\it Chern-scalar curvature} is the trace
\begin{equation}\label{eq:def-scal}
\scal^{\rm Ch}(g) \in \cC^{\infty}(M,\bR) \,\, , \quad \scal^{\rm Ch}(g) \= 2\Tr^{\bC}_g(\tilde{S}(g)) \,\, .
\end{equation}
We remark that, with this notation, when $g$ is K\"ahler it holds that
$$
S(g) = \Ric(g) \,\, , \quad \scal^{\rm Ch}(g)= \scal(g) \,\, ,
$$
where $\Ric(g)$ and $\scal(g)$ denote the Riemannian Ricci endomorphism and the Riemannian scalar curvature of $g$, respectively. We recall that $g$ is called {\em first-Chern-Einstein} if it satisfies $S(g) =\tfrac{\l}{2m} \Id$ for some $\l \in\cC^{\infty}(M,\bR)$, see \cite{tosatti, angella-calamai-spotti-2} and references therein. Notice that, in this case $\l = \scal^{\rm Ch}(g)$ and, if $g$ is K\"ahler, then this notion corresponds to the K\"ahler-Einstein condition.

We also set $\diff^{\,c} \= J^{-1} \! \circ \diff\, \circ J$, so that
$$
\diff = \p+\bar{\p} \,\, , \quad
\diff^{\,c} = -\ti(\p-\bar{\p}) \,\, , \quad
\diff\diff^{\,c}=2\ti\p\bar{\p}
$$
and we denote by $\D_g \= (D^g)^*D^g$ the {\it Laplace-Beltrami operator}. We recall that $\diff\diff^{\,c}$ and $\D_g$ are related by the following

\begin{lemma}[{see e.g. \cite[p. 502]{gauduchon-mathann}}]
For any function $u \in \cC^{\infty}(M,\bR)$ it holds that
\begin{equation} \label{ChernLapl}
\Tr^{\bC}_g(\diff\diff^{\,c}\!u) = \D_g u + g(\diff u, \q) \,\, .
\end{equation}
\end{lemma}

\begin{proof}
For the sake of completeness, we summarize here the computation.
By the very definition, for any vector fields $X,Y$ it holds that
\begin{equation} \label{eq:ddcu}
\diff\diff^{\,c}\!u(X,Y) = \eL_{X}\eL_{JY}u - \eL_{Y}\eL_{JX}u - \eL_{J[X,Y]}u \,\, .
\end{equation}
Let now $(\tilde{e}_{\a})=(e_i,Je_i)$ be a local $(J,g)$-unitary frame on $M$ and set $A \= \n - D^g$. Then, notice that
$$
\sum_{1 \leq i \leq m} -J[e_i,Je_i] = \sum_{1 \leq i \leq m} -J(\n_{e_i}Je_i -\n_{Je_i}e_i -T(e_i,Je_i)) = \sum_{1 \leq \a \leq 2m} \n_{\tilde{e}_{\a}}\tilde{e}_{\a}
$$
and so
$$\begin{aligned}
\Tr^{\bC}_g(\diff\diff^{\,c}u) &= \sum_{1\leq i \leq m} (\diff\diff^{\,c}u)(e_i,Je_i) \\
&= -\sum_{1\leq i \leq m}(\eL_{e_i}\eL_{e_i}u +\eL_{Je_i}\eL_{Je_i}u) -\sum_{1\leq i \leq m}\eL_{J[e_i,Je_i]}u \\
&= \sum_{1\leq \a \leq 2m} \!\!-(\eL_{\tilde{e}_{\a}}\eL_{\tilde{e}_{\a}} -\eL_{D^g_{\tilde{e}_{\a}}\tilde{e}_{\a}})u +\sum_{1\leq \a \leq 2m}\eL_{A(\tilde{e}_{\a},\tilde{e}_{\a})}u \,\, .
\end{aligned}$$
Moreover, from \eqref{defChern} we get
\begin{align*}
\sum_{1\leq \a \leq 2m} A(\tilde{e}_{\a},\tilde{e}_{\a}) &= \sum_{1\leq \a,\b \leq 2m}  -\tfrac12\diff\w(J\tilde{e}_{\a},\tilde{e}_{\a},\tilde{e}_{\b})\tilde{e}_{\b} \\
&= \sum_{1\leq \b \leq 2m} \sum_{1\leq i \leq m} \diff\w(\tilde{e}_{\b},e_i,Je_i)\tilde{e}_{\b} = \sum_{1\leq \b \leq 2m} \q(\tilde{e}_{\b})\tilde{e}_{\b} = \q^{\#}
\end{align*}
and therefore we obtain \eqref{ChernLapl}.
\end{proof}

Given a function $u \in \cC^{\infty}(M,\bR)$, we also denote by $\Hess_g(u) \in \cC^{\infty}(M,\Sym(TM))$ the {\it Hessian of $u$} defined by
$$
g(\Hess_g(u)(X),Y) \= (D^g_X(\diff u))(Y)
$$
and we stress the following

\begin{lemma}
If $g$ is K\"ahler, then
\begin{equation} \label{ddcKahler}
\diff\diff^{\,c}\!u(X,Y) = g(\Hess_g(u)(X),JY) -g(\Hess_g(u)(JX),Y) \,\, .
\end{equation}
\end{lemma}

\begin{proof}
By the very definition, we have
$$
g(\Hess_g(u)(X),Y) = \eL_{X}\eL_{Y}u -\eL_{D^g_XY}u \,\, .
$$
Therefore, since $g$ is K\"ahler, we get
$$\begin{aligned}
g(\Hess_g(u)(X),&JY) -g(\Hess_g(u)(JX),Y) \\
&= \eL_{X}\eL_{JY}u -\eL_{D^g_X(JY)}u -\eL_{JX}\eL_{Y}u +\eL_{D^g_{JX}Y}u \\
&= \eL_{X}\eL_{JY}u -\eL_{Y}\eL_{JX}u -\eL_{[JX,Y]}u -\eL_{JD^g_XY}u +\eL_{JD^g_YX}u +\eL_{[JX,Y]}u \\
&= \eL_{X}\eL_{JY}u - \eL_{Y}\eL_{JX}u - \eL_{J[X,Y]}u \\
&= \diff\diff^{\,c}\!u(X,Y)
\end{aligned}$$
which concludes the proof.
\end{proof}

Finally, we introduce the following two operators
$$
\d_g, \d^{\n} : \cC^{\infty}(M,\Sym(TM)) \to \cC^{\infty}(M,T^*M)
$$
defined by
\begin{equation} \label{formula:div}
(\d_gh)(X) \= \tr^{\bR}\!\big((D^gh)X\big) \,\, , \quad (\d^{\n}h)(X) \= \tr^{\bR}\!\big((\n h)X\big) \,\, .
\end{equation}

\section{Variation formulas for the Chern-scalar curvature} \label{sec:variation}
\setcounter{equation} 0

Let $(M,J)$ be a compact, connected, complex manifold of dimension $\dim_{\bR}M=2m$. We consider the bundle $S^{1,1}_+(T^*M) \to M$ of symmetric, bilinear, positive definite, $J$-invariant forms and, for any integer $k \geq m+1$, we define the space of {\it Hermitian metrics on $(M,J)$ of class $W^{2,k}$} as
$$
\eM_{\rm H}^k \= W^{k,2}(M,S^{1,1}_+(T^*M)) \,\, .
$$
For the sake of notation, we denote the space of smooth Hermitian metrics by $\eM_{\rm H} \= \cC^{\infty}(M,S^{1,1}_+(T^*M))$.

\subsection{First variation of the Chern-scalar curvature} \hfill \par

Fix a smooth Hermitian metric $g \in \eM_{\rm H}$. Then, given an element $h \in \cC^{\infty}(M,\Sym^{1,1}(TM))$, we consider the corresponding path $(g_t) \subset \eM_{\rm H}$ given by
$$
g_t \= g((\Id+th)\,\cdot,\cdot\cdot) \,\, , \quad t \in (-\epsilon,\epsilon)
$$
with $\epsilon>0$ small enough. For the sake of notation, in this section we will always use this shortener notation: if $F$ is a function defined on $\eM_{\rm H}$, we write $F'$ instead of $F'_g(h)$ to denote the differential of $F$ at $g$ in the direction of $h$.

\begin{prop} The differential at $g$ in the direction of $h$ of the Chern connection  is the $(1,2)$-tensor field $\n'$ defined by
\begin{equation} \label{nablalin}
2\,\n'_XY = (\n_Xh)(Y) -(J\circ\n_{JX}h)(Y) \,\, .
\end{equation}
\end{prop}

\begin{proof}
We set $C^t \= \n^{g_t}-\n^g$, that is a $(2,1)$-tensor field. Then, by using the Koszul Formula of the Chern connection (see e.g. \cite[Sect 2.1]{angella-pediconi})
\begin{align*}
2g_t(C^t_XY,Z) &= 2g_t(\n^{g_t}_XY,Z) -2g_t(\n^g_XY,Z) \\
&= \eL_X(g_t(Y,Z)) -\eL_X(g(Y,Z)) -\eL_{JX}(g_t(JY,Z)) +\eL_{JX}(g(JY,Z)) \\
&\phantom{=} +g_t([X,Y],Z) -g([X,Y],Z) -g_t([JX,Y],JZ) +g([JX,Y],JZ) \\
&\phantom{=} -g_t([X,Z],Y) +g([X,Z],Y) +g_t([JX,Z],JY) -g([JX,Z],JY) -2tg(h(\n^g_XY),Z) \\
&=t \eL_X(g(h(Y),Z)) -t\eL_{JX}(g(h(JY),Z)) +tg(h([X,Y]),Z) -tg(h([JX,Y]),JZ) \\
&\phantom{=} -tg(h([X,Z]),Y) +tg(h([JX,Z]),JY) -2tg(h(\n^g_XY),Z)
\end{align*}
and so
\begin{align*}
2&g((\tfrac1tC^t)_XY,Z) +2g(h(C^t_XY),Z) = \\
&= g((\n^g_Xh)(Y),Z) -g((\n^g_{JX}h)(JY),Z) -g(h(\n^g_XY),Z) +g(h([X,Y]),Z) +g(h(\n^g_XZ),Y) \\
&\phantom{=} -g(h([X,Z]),Y) +g(h(\n^g_{JX}Y),JZ) -g(h([JX,Y]),JZ) -g(h(\n^g_{JX}Z),JY) +g(h([JX,Z]),JY) \\
&= g((\n^g_Xh)(Y),Z) -g((\n^g_{JX}h)(JY),Z) -g(h(\n^g_YX+T^g(X,Y)),Z) +g(h(\n^g_YX+T^g(X,Y)),Z) \\
&\phantom{=} +g(h(\n^g_ZX+T^g(X,Z)),Y) -g(h(\n^g_ZX+T^g(X,Z)),Y) \\
&= g((\n^g_Xh)(Y),Z) -g((\n^g_{JX}h)(JY),Z) \,\, .
\end{align*}
Letting $t\to0$, we obtain \eqref{nablalin}.
\end{proof}

\begin{lemma}
The differential at $g$ in the direction of $h$ of the complex trace is the linear map defined, for $\alpha \in \cC^{\infty}(M,\L^{1,1}(T^*M))$, by
\begin{equation} \label{eq:diffTrC}
(\Tr^{\bC})'(\a) = -\tfrac12 g(h,\rho_g^{-1}(\a)) \,\, .
\end{equation}
\end{lemma}
\begin{proof}
Fix a local $(J,g)$-unitary frame $(e_i,Je_i)$ on $M$ and the associated frame
$$
\e_i \= \tfrac1{\sqrt2}(e_i-\ti Je_i) \,\, , \quad \e_{\bar{i}} \= \tfrac1{\sqrt2}(e_i+\ti Je_i) \,\, ,
$$
so that
$$
g = \d_{\bar{j}i} \, \e^i \odot \e^{\bar{j}} \,\, , \quad 
\rho_g(h) = h_{\bar{j}i} \,\ti\, \e^i \wedge \e^{\bar{j}} \,\, , \quad
\a = \a_{\bar{j}i} \,\ti\, \e^i \wedge \e^{\bar{j}} \,\, ,
$$
where $h_{\bar{j}i} = \d_{\bar{j}s}h^s_{{\phantom{s}}i}$. If $(g_t)_{\bar{j}i} \= \d_{\bar{j}i} +th_{\bar{j}i}$ and $((g_t)^{i\bar{j}}) \= ((g_t)_{\bar{j}i})^{-1}$, it follows that
$$
(g_t)^{i\bar{j}}= \d^{i\bar{j}} -t g^{i\bar{r}}h_{\bar{r}s}(g_t)^{s\bar{j}} \,\, .
$$
Since $\Tr^{\bC}_g(\a) = \d^{i\bar{j}}\a_{\bar{j}i}$, we get
$$
(\Tr^{\bC})'(\a) = \lim_{t\to0}\tfrac1t(\Tr^{\bC}_{g_t}(\a)-\Tr^{\bC}_g(\a)) = 
\lim_{t\to0}\tfrac1t\big((g_t)^{i\bar{j}}-\d^{i\bar{j}}\big)\a_{\bar{j}i} = 
-h^{i\bar{j}}\a_{\bar{j}i} \,\, ,
$$
where $h^{i\bar{j}}= h^i_{{\phantom{i}}s}\d^{s\bar{j}}$. On the other hand, one can directly check that
$$
g(h,\rho_g^{-1}(\a)) = 2h^{i\bar{j}}\a_{\bar{j}i} \,\, ,
$$
which completes the proof. \end{proof}

\begin{prop}
The differential at $g$ in the direction of $h$ of the Chern-curature operator is given by
\begin{equation} \label{Rlin}
2\,\W'(X,Y) = [\W(X,Y),h] +J \circ \big(\n_X\n_{JY}h -\n_Y\n_{JX}h -\n_{J[X,Y]}h\big) \,\, .
\end{equation}
\end{prop}

\begin{proof}
By differentiating \eqref{Chern-curv}, we get
$$
\W'(X,Y) = \n'_{[X,Y]}-[\n_X,\n'_Y]-[\n'_X,\n_Y] \,\, .
$$
Notice that
$$
2[\n_X,\n'_Y] = \n_X\n_Yh -J \circ \n_X\n_{JY}h
$$
and so we get
$$\begin{aligned}
2\,\W'(X,Y) &= \n_{[X,Y]}h - J \circ \n_{J[X,Y]}h - \n_X\n_Yh + J \circ \n_X\n_{JY}h + \n_Y\n_Xh - J \circ \n_Y\n_{JX}h \\
&= \W_{X,Y}h +J \circ \big(\n_X\n_{JY}h -\n_Y\n_{JX}h -\n_{J[X,Y]}h\big) \,\, .
\end{aligned}$$
Here, by a slight abuse of notation, we denoted by $\W_{X,Y}$ the curvature of the connection naturally induced on the bundle $\End(TM)$.
Since $\W_{X,Y}h = [\W(X,Y),h]$, we obtain \eqref{Rlin}.
\end{proof}

\begin{lemma}
For any vector field $X$ it holds that
\begin{equation} \label{dertrace}
\tr^{\bC}(\n_Xh)=\eL_X(\tr^{\bC}h) \,\, .
\end{equation}
\end{lemma}

\begin{proof}
Fix $x \in M$ and let $(\tilde{e}_{\a})$ be a local orthonormal frame around $x$ such that $(D\tilde{e}_{\a})_x=0$. Then, by using \eqref{defChern}, at the point $x$ we get
$$\begin{aligned}
\tr^{\bR}(\n_Xh) &= \sum_{\a} g((\n_Xh)(\tilde{e}_{\a}),\tilde{e}_{\a}) \\
&= \sum_{\a} g(\n_X(h(\tilde{e}_{\a})),\tilde{e}_{\a}) -g(h(\n_X\tilde{e}_{\a}),\tilde{e}_{\a}) \\
&= \sum_{\a} \eL_Xg(h(\tilde{e}_{\a}),\tilde{e}_{\a}) -2g(h(\tilde{e}_{\a}),\n_X\tilde{e}_{\a}) \\
&= \eL_X(\tr^{\bR}h)+\sum_{\a,\b}\diff\w(JX,\tilde{e}_{\a},\tilde{e}_{\b})g(h(\tilde{e}_{\a}),\tilde{e}_{\b}) \\
&= 2\eL_X(\tr^{\bC}h) \,\,.
\end{aligned}$$
Moreover, since $\n_Xh \in \cC^{\infty}(M,\Sym^{1,1}(TM))$, it holds that
$$\tr^{\bR}(\n_Xh) = 2\tr^{\bC}(\n_Xh)$$
and so we get the thesis.
\end{proof}

\begin{prop}
The differential at $g$ in the direction of $h$ of the first Chern-Ricci form is given by
\begin{equation} \label{S1lin}
\tilde{S}{}' = \tfrac12\diff\diff^{\,c}(\tr^{\bC}h) \,\, .
\end{equation}
\end{prop}

\begin{proof}
Notice that the trace $\tr^{\bC}$ commutes with the differentiation. Moreover, it is straightforward to check that, for any vector fields $X,Y$ and $h \in \cC^{\infty}(M,\Sym^{1,1}(TM))$, the commutator $[\W(X,Y),h]$ is a section of $\Sym^{1,1}(TM)$, with zero trace. Therefore, $J \circ [\W(X,Y),h] \in \cC^{\infty}(M,\Skew^{1,1}(TM))$, and so, by means of \eqref{eq:def-tildeS}, \eqref{Rlin}, \eqref{dertrace}, and \eqref{eq:ddcu}, we obtain
$$\begin{aligned}
\tilde{S}{}'(X,Y) &= \tfrac12\tr^{\bC}\!\big(\n_X\n_{JY}h -\n_Y\n_{JX}h -\n_{J[X,Y]}h\big) \\
&= \tfrac12\big(\eL_X\eL_{JY}(\tr^{\bC}h) -\eL_Y\eL_{JX}(\tr^{\bC}h) -\eL_{J[X,Y]}(\tr^{\bC}h)\big) \\
&= \tfrac12\diff\diff^{\,c}(\tr^{\bC}h)(X,Y)
\end{aligned}$$
and so we get \eqref{S1lin}.
\end{proof}

As a direct consequence of \eqref{eq:diffTrC} and \eqref{S1lin}, and by \eqref{ChernLapl}, we get

\begin{prop}\label{prop:first-variation-scal}
The differential at $g$ in the direction of $h$ of the Chern-scalar curvature is given by
\begin{equation} \label{slin}
(\scal^{\rm Ch})' = \D_g(\tr^{\bC}h)+g(\diff\,(\tr^{\bC}h),\q) -g(h,S(g)) \,\, .
\end{equation}
\end{prop}

\subsection{Second variation of the Chern-scalar curvature} \hfill \par

We are going to compute the second variation of the Chern-scalar curvature at a fixed point $g \in \eM_{\rm H}$. Let $g_t = g((\Id+th)\,\cdot,\cdot\cdot)$ as before. From \cite{besse}, and by recalling the definition of $\d_g$ in \eqref{formula:div}, we have

\begin{prop}[{see \cite[Prop 1.184]{besse}}]
The differential at $g$ in the direction of $h$ of the Laplace-Beltrami operator is given by
\begin{equation} \label{Deltalin}
\D'u = g(\Hess_g(u),h) +g(\diff u, \d_gh) -g(\diff u, \diff\,(\tr^{\bC}h)) \,\, .
\end{equation}
\end{prop}

By means of a straightforward computation in local coordinates, one can show that, for any $\a,\b \in \cC^{\infty}(M,T^*M)$ and $A,B \in \cC^{\infty}(M,\End(TM))$, we have
\begin{equation} \label{eq:lemmadiffg}
g(\a,\b)' = -g(\a \circ h,\b) \,\, , \quad g(A,B)' = 0 \,\, .
\end{equation}

For the following, we recall that $\d^{\n}$ has been defined in \eqref{formula:div}:
\begin{prop}
The differential at $g$ in the direction of $h$ of the Lee form is given by
\begin{equation} \label{diffLee}
\q' = \diff\,(\tr^{\bC}h) - \d^{\n}h \,\, .
\end{equation}
\end{prop}

\begin{proof}
By the very definition of Lee form \eqref{eq:def-lee}, we have
$$
\q'(X) = \tr^{\bR}\!\big(T'(X,\cdot)\big) =\tr^{\bR}\!\big(\n'_X(\cdot)\big) -\tr^{\bR}\!\big(\n'_{(\cdot)}X\big) \,\, .
$$
Then, from \eqref{nablalin} and \eqref{dertrace}, it follows that
$$
\tr^{\bR}\!\big(\n'_X(\cdot)\big) = \tfrac12\tr^{\bR}\!\big(\n_Xh\big) -\tfrac12\tr^{\bR}\!\big(J \circ (\n_{JX}h)\big) = \eL_X(\tr^{\bC}h) \,\, .
$$
Moreover, given a local $(J,g)$-unitary frame $(e_i,Je_i)$, from \eqref{nablalin} and \eqref{formula:div} we get
$$\begin{aligned}
\tr^{\bR}\!\big(\n'_{(\cdot)}X\big) &= \tfrac12 \sum_{1 \leq i \leq m} \big(g((\n_{e_i}h)(X),e_i) +g((\n_{Je_i}h)(X),Je_i) \\
&\hskip 55pt  -g((\n_{Je_i}h)(JX),e_i) +g((\n_{e_i}h)(JX),Je_i)\big) \\
&= \sum_{1 \leq i \leq m} g((\n_{e_i}h)(X),e_i) +g((\n_{Je_i}h)(X),Je_i) \\
&=(\d^{\n}h)(X)
\end{aligned}$$
and so this proves Formula \eqref{diffLee}.
\end{proof}

\begin{prop}
The differential at $g$ in the direction of $h$ of the first Chern-Ricci symmetric endomorphism is given by
\begin{equation} \label{S1symlin}
g(S{}'(X),Y) = -g((h \circ S(g))(X),Y) +\tfrac12\diff\diff^{\,c}(\tr^{\bC}h)(X,JY) \,\, .
\end{equation}
\end{prop}

\begin{proof}
Differentiating both sides of $g(S(g)(X),Y) = \tilde{S}(g)(X,JY)$, we get
$$
g((h \circ S(g))(X),Y) +g(S{}'(X),Y) = \tilde{S}'(X,JY)\,\, .
$$
Therefore, by using \eqref{S1lin}, we get \eqref{S1symlin}.
\end{proof}

Finally, we are ready to prove the following

\begin{prop}\label{prop:second-variation-scal}
The second differential at $g$ in the direction of $(h_1,h_2)$ of the Chern-scalar curvature is given by
\begin{multline} \label{slin2}
(\scal^{\rm Ch})'' = 
g(\Hess_g(\tr^{\bC}h_1),h_2) +g(\diff\,(\tr^{\bC}h_1), \d_gh_2-\d^{\n}h_2)
-g(\diff\,(\tr^{\bC}h_1),\q \circ h_2) \\
+g(h_1,h_2\circ S(g)) -\tfrac12g(h_1,\rho_g^{-1}(\diff\diff^{\,c}(\tr^{\bC}h_2))) \,\, .
\end{multline}
Moreover, if $g$ is K\"ahler, then
\begin{equation} \label{slin2K}
(\scal^{\rm Ch})'' = 
g(\Hess_g(\tr^{\bC}h_1),h_2) +g(h_1,\Hess_g(\tr^{\bC}h_2)) +g(h_1,h_2\circ \Ric(g)) \,\, .
\end{equation}
\end{prop}

\begin{proof}
Notice that \eqref{slin2} follows directly from \eqref{slin}, \eqref{Deltalin}, \eqref{eq:lemmadiffg}, \eqref{diffLee} and \eqref{S1symlin}. Assume now that $g$ is K\"ahler, then 
$$\d^{\n}=\d_g \,\, , \quad \q=0 \,\, , \quad S(g)=\Ric(g) \,\, .$$
Moreover, by \eqref{ddcKahler}, we get
$$\begin{aligned}
-\tfrac12g&(h_1,\rho_g^{-1}(\diff\diff^{\,c}(\tr^{\bC}h_2))) \\
&= -\tfrac12 \sum_{1 \leq \a,\b \leq 2m} g(h_1(\tilde{e}_{\a}),\tilde{e}_{\b})\diff\diff^{\,c}(\tr^{\bC}h_2)(\tilde{e}_{\a},J\tilde{e}_{\b}) \\
&= \tfrac12 \sum_{1 \leq \a,\b \leq 2m} \big(g(h_1(\tilde{e}_{\a}),\tilde{e}_{\b})g(\Hess_g(\tr^{\bC}h_2)(\tilde{e}_{\a}),\tilde{e}_{\b}) +g(h_1(\tilde{e}_{\a}),\tilde{e}_{\b})g(\Hess_g(\tr^{\bC}h_2)(J\tilde{e}_{\a}),J\tilde{e}_{\b})\big) \\
&= \sum_{1 \leq \a,\b \leq 2m} g(h_1(\tilde{e}_{\a}),\tilde{e}_{\b})g(\Hess_g(\tr^{\bC}h_2)(\tilde{e}_{\a}),\tilde{e}_{\b}) \\
&= g(h_1,\Hess_g(\tr^{\bC}h_2))
\end{aligned}$$
which concludes the proof of \eqref{slin2K}.
\end{proof}

\section{Main results}\label{sec:main}
\setcounter{equation} 0

Let again $(M,J)$ be a compact, connected, complex manifold of dimension $\dim_{\bR}M=2m$. In this section, we prove our main results, concerning the submersion and local surjectivity properties for the map $\scal^{\rm Ch} : \eM_{\rm H} \to \cC^{\infty}(M,\bR)$, in view of the relationship between infinitesimal and actual deformations of the Chern-scalar curvature function with respect to a varying metric (see \cite{fischer-marsden} for the Riemannian case).

\subsection{Submersion points for the Chern-scalar curvature} \hfill \par

We begin this section with the following striaghtforward

\begin{lemma}
For any integer $k\geq m+2$, the map $\scal^{\rm Ch} : \eM^{k+2}_{\rm H} \to W^{k,2}(M,\bR)$ is smooth.
\end{lemma}
\begin{proof}
This follows directly by the local formula of the Chern-scalar curvature.
More precisely, take a local chart $(\eU, \xi=(z^1,{\dots},z^m))$ of holomorphic coordinates in an open set of $(M,J)$. If we denote by $\G$ the Christoffel symbol of $\n$ with respect to $(\eU, \xi)$, by using the Koszul Formula for the Chern connection (see e.g. \cite[Sect 2.1]{angella-pediconi} for notation), one can directly check that the only non-vanishing components of $\G$ are
\begin{equation} \label{ChristChern}
\G_{ij}^k = g^{k\bar{r}}g_{\bar{r}j,i} \,\, , \quad \G_{\bar{i}\bar{j}}^{\bar{k}} = \ol{\G_{ij}^k} \,\, .
\end{equation}
Then, if we denote by $\hat{\W}(g) \in \cC^{\infty}(M,\L^{1,1}(T^*M) \otimes \L^{1,1}(T^*M))$ the totally covariant Chern-curvature, by using \eqref{ChristChern} we can write
$$\begin{gathered}
\hat{\W}(g) = \W_{\bar{j}i\bar{\ell}k} \, \ti \diff z^{\bar{j}} \wedge \diff z^{i} \otimes \ti \diff z^{k} \wedge \diff z^{\bar{\ell}} \,\, , \quad \W_{\bar{j}i\bar{\ell}k} = -\hat{\W}\big(\tfrac{\p}{\p z^i},\tfrac{\p}{\p z^{\bar{j}}},\tfrac{\p}{\p z^k},\tfrac{\p}{\p z^{\bar{\ell}}}\big)= -g_{\bar{\ell}k,\bar{j}i}+g^{s\bar{r}}g_{\bar{r}k,i}g_{\bar{\ell}s,\bar{j}} \,\, .
\end{gathered}$$
Therefore, by the very definition of Chern-scalar curvature, we get
\begin{equation} \label{locscal}
\scal^{\rm Ch}(g) = 2g^{k\bar{\ell}}g^{i\bar{j}}\W_{\bar{j}i\bar{\ell}k} = -2g^{k\bar{\ell}}g^{i\bar{j}}g_{\bar{\ell}k,\bar{j}i} +2g^{k\bar{\ell}}g^{i\bar{j}}g^{s\bar{r}}g_{\bar{r}k,i}g_{\bar{\ell}s,\bar{j}} \,\, .
\end{equation}
Finally, notice that the local formula \eqref{locscal} and the multiplicative properties of the Sobolev spaces (see e.g. \cite[Thm 4.39]{adams}) imply that the map $\scal^{\rm Ch} : \eM^{k+2}_{\rm H} \to W^{k,2}(M,\bR)$ is smooth for any $k \geq m+2$.
\end{proof}

For any $g \in \eM_{\rm H}$, we consider the linearized Chern-scalar curvature
$$
\g_g: \cC^{\infty}(M,\Sym^{1,1}(TM)) \to \cC^{\infty}(M,\bR) \,\, , \quad \g_g(h) = \tfrac12 \big(\D_g(\tr^{\bR}h)+g(\diff\,(\tr^{\bR}h),\q)\big) -g(h,S)
$$
given by \eqref{slin} and its $L^2$-adjoint
\begin{equation} \label{eq:gamma*}
\g_g^*: \cC^{\infty}(M,\bR) \to \cC^{\infty}(M,\Sym^{1,1}(TM)) \,\, , \quad \g_g^*(u) = \tfrac12\big(\D_gu-g(\diff u, \q) +(\diff^{\,*}\!\q)\,u\big)\Id -u\,S \,\, .
\end{equation}

For the following result in the Riemannian context, compare \cite[Thm 1]{fischer-marsden}.

\begin{prop} \label{prop:submersion}
Fix an integer $k\geq m+2$, let $g \in \eM^{k+2}_{\rm H}$ and set $\l \= \scal^{\rm Ch}(g) \in W^{k,2}(M,\bR)$. Assume that one of the following is satisfied: \begin{itemize}
\item[i)] $g$ is not first-Chern-Einstein;
\item[ii)] $g$ is first-Chern-Einstein and $\tfrac{\l}m<\diff^{\,*}\!\q$.
\end{itemize}
Then, the map $\scal^{\rm Ch} : \eM^{k+2}_{\rm H} \to W^{k,2}(M,\bR)$ is a submersion at $g$.
\end{prop}

\begin{proof}
The principal symbol of $\g_g^*$ is clearly injective. Therefore, by means of Remark \ref{rem:adjP}, it is sufficient to prove that both (i), (ii) imply that $\g_g^*$ is injective.

Firstly, assume that $g$ is not first-Chern-Einstein. Take $u \in \ker(\g_g^*)$ and a vector field $X \in \cC^{\infty}(M,TM)$ such that $g(X,X)=1$ and set $g(S(g)(X),X)=\tfrac{\phi}{2m}$, for some $\phi \in W^{k,2}(M,\bR)$. Then, by \eqref{eq:gamma*} we get
\begin{equation} \label{eq:gammaulambda}
0 = 2g(\g_g^*(u)X,X) = \D_gu-g(\diff u, \q) +\big((\diff^{\,*}\!\q)-\tfrac{\phi}{m}\big)u\,\, .
\end{equation}
Since $k\geq m+2$, by \eqref{embSob} the coefficients of \eqref{eq:gammaulambda} are of class $\cC^1$. We then apply \cite[Thm 1.17]{cheeger-naber-valtorta}: it follows that $u \equiv 0$ or there exists an open dense subset $\eU \subset M$ such that $u(x) \neq 0$ for any $x \in \eU$ and the complement $M \setminus \eU$ has zero measure. Assume by contradiction that $u \not\equiv 0$. Then, by \eqref{eq:gamma*} and \eqref{eq:gammaulambda}, we obtain
$$
0=\g_g^*(u) = u(-S(g)+\tfrac{\phi}{2m}\Id)
$$
and so
$$
S(g)=\tfrac{\phi}{2m}\Id \quad \text{ for any } x \in \eU \,\, .
$$
In particular, this implies that $\phi = \l$ and that $g$ is first-Chern-Einstein on the whole manifold $M$, which is not possible by assumption. Then, (i) implies that $\g_g^*$ is injective.

On the other hand, assume that $S(g) = \tfrac{\l}{2m}\Id$. Then, the equation $\g_g^*(u)=0$ is equivalent to
$$
\D_gu-g(\diff u, \q) +((\diff^{\,*}\!\q)-\tfrac{\l}{m})u = 0 \,\, .
$$
If (ii) holds true, then the Strong Maximum Principle (see e.g. \cite[Thm 4.2]{goffi-pediconi}) implies that $u \equiv 0$. \end{proof}

From Proposition \ref{prop:submersion}, we get the proofs of Theorem \ref{thm:a} and Theorem \ref{thm:b}.

\begin{proof}[Proof of Theorem \ref{thm:a}]
Let $g \in \eM_{\rm H}$, set $\l \= \scal^{\rm Ch}(g) \in \cC^{\infty}(M,\bR)$ and assume that either $g$ is not first-Chern-Einstein, or $g$ is first-Chern-Einstein with $\tfrac{\l}m<\diff^{\,*}\!\q$. Thanks to Proposition \ref{prop:submersion} and the Implicit Function Theorem, we know that, for any integer $k \geq m+2$, the extension of $\scal^{\rm Ch}$ on $\eM_{\rm H}^{k+2}$ is linearization stable at $g$ and maps any neighborhood of $g$ in $\eM^{k+2}_{\rm H}$ onto a neighborhood of $\l$ in $W^{k,2}(M,\bR)$. Moreover, by means of Theorem \ref{thm:BE},
$$
W^{k,2}(M,\Sym^{1,1}(TM)) = \ker(\g_g) \oplus \im(\g_g^*) \,\, .
$$
It remains to prove the following claim: if $h \in \im(\g_g^*) \subset W^{k,2}(M,\Sym^{1,1}(TM))$ and $g_t := g((\Id+th)\,\cdot,\cdot\cdot)$, then $\scal^{\rm Ch}(g_t) \in \cC^{\infty}(M,\bR)$ only if $h \in \cC^{\infty}(M,\Sym^{1,1}(TM))$.

So, take $h \in \im(\g_g^*) \subset W^{k,2}(M,\Sym^{1,1}(TM))$ and assume that $\scal^{\rm Ch}(g_t) \in \cC^{\infty}(M,\bR)$. By diffe\-rentia\-ting with respect to $t$, it follows that $\g_g(h) \in \cC^{\infty}(M,\bR)$. By hypothesis, there exists $u \in W^{k+2,2}(M,\bR)$ such that $h=\g_g^*(u)$. Since $\g_g\g_g^*$ has injective symbol and $\g_g\g_g^*(u) \in \cC^{\infty}(M,\bR)$, it follows that $u \in \cC^{\infty}(M,\bR)$. Therefore $h \in \cC^{\infty}(M,\Sym^{1,1}(TM))$ and the claim follows.
\end{proof}

\begin{proof}[Proof of Theorem \ref{thm:b}]
Let us observe that, if one of the conditions stated in Theorem \ref{thm:b} is satisfied, then $(M,J)$ does not admit any smooth first-Chern-Einstein metric (see \cite[Thm 5]{angella-calamai-spotti-2}).
Take $g\in\eM_{\rm H}(\l)$.
Then, for any $k\geq m+2$, by means of Proposition \ref{prop:submersion} and the Implicit Function Theorem, the preimage $(\scal^{\rm Ch})^{-1}(\l)$ inside $\eM_{\rm H}^{k+2}$ is a smooth submanifold in a neighborhood of $g$ with tangent space at $g$ given by $\ker(\g_g) \subset W^{k+2,2}(M,\Sym^{1,1}(TM))$. This gives rise to a structure of smooth ILH-submanifold on the space $\eM_{\rm H}(\l) = \{g \in \eM_{\rm H} : \scal^{\rm Ch}(g)=\l\}$ inside $\eM_{\rm H}$.
\end{proof}

\begin{rem}\label{rmk:non-empty}
First, we recall that, by \cite[Thm 1]{gauduchon-excentricite}, in any conformal class of Hermitian metrics, there is a unique {\em Gauduchon metric} with volume $1$, where being Gauduchon means that the associated $(1,1)$-form $\h$ satisfies $\diff\diff^{\,c}\eta^{m-1}=0$. The {\em Gauduchon degree} of the conformal class $\{\eta\}$ is then defined as
$$ \G(\{\eta\}) := \int_M c_1^{\mathsmaller{\rm BC}}(K_M^{-1})\wedge \eta^{m-1} = \int_M \scal^{\rm Ch}(\eta) \,\eta^m \,\, , $$
and it is equal to the volume of the divisor associated to
any meromorphic section of the anti-canonical line bundle $K_M^{-1}$ by means of the Gauduchon metric, see \cite[Sect I.17]{gauduchon-mathann}.
We can now recall the following known facts concerning the condition for $\eM_{\rm H}(\l)$ to be non-empty.
\begin{itemize}
\item[$\bcdot$] If there exists a conformal class $\{\eta\}$ on $(M,J)$ with $\G(\{\eta\}) = 0$, then $\eM_{\rm H}(\l)$ is non-empty for $\lambda = 0$ constant (see \cite[Thm 3.1]{angella-calamai-spotti}). Moreover, if $\eta$ is balanced (that is, $\diff\eta^{m-1}=0$) with $\scal^{\rm Ch}(\eta) = 0$, then $\eM_{\rm H}(\l)$ is non-empty if $\lambda\in\cC^{\infty}(M,\bR)$ changes sign and $\int_M \lambda \, \eta^m<0$ (see \cite[Thm 2.11]{fusi}).
\item[$\bcdot$] If there exists a conformal class $\{\eta\}$ on $(M,J)$ with $\G(\{\eta\})<0$, then $\eM_{\rm H}(\l)$ is non-empty for any negative constant $\lambda<0$ (see \cite[Thm 4.1]{angella-calamai-spotti}) and for any non-identically zero $\lambda \in \cC^{\infty}(M,\bR)$ such that $\lambda\leq0$ (see \cite[Thm 2.5]{fusi}). This happens, in particular, when the Kodaira dimension ${\rm Kod}(M,J)$ is positive (see \cite[Sect I.17]{gauduchon-mathann}).
\item[$\bcdot$] If $M=N \times \Sigma$, where $N$ is a compact complex manifold  admitting a conformal class $\{\eta\}$ with $\G(\{\eta\})>0$ and $\Sigma$ is a compact Riemann surface with $\chi(\S)<0$, then $\eM_{\rm H}(\l)$ is non-empty for any positive constant $\lambda>0$ (see \cite[Prop 5.7]{angella-calamai-spotti}).
\item[$\bcdot$] If the Chern-Yamabe conjecture \cite[Conj 2.1]{angella-calamai-spotti} has an affirmative answer, then $\eM_{\rm H}(\l)$ is non-empty for $\lambda \in I \subset \bR$ constant as follows: $I = \bR$ when neither $K_M$ nor $K_M^{-1}$ is pseudo-effective; $I = (0,+\infty)$ when $K_M^{-1}$ is pseudo-effective and non-unitary flat; $I = (-\infty,0)$ when $K_M$ is pseudo-effective and non-unitary flat; $I =\{0\}$ when $K_M$ is unitary flat (see \cite{teleman}, \cite[Thms 1.1, 3.4]{yang}).
\item[$\bcdot$] Further examples of compact manifolds admitting Hermitian metrics with positive constant Chern-scalar curvature are given by the Hopf surface \cite{gauduchon-ivanov}, the homogeneous non-K\"ahler ${\rm C}$-spaces \cite{podesta}, the Hirzebruch surfaces \cite{koca-lejmi}, the B\'erard-Bergery standard cohomogeneity one complex manifolds \cite{angella-pediconi}.
\end{itemize}
\end{rem}

\subsection{Linearization instability and infinitesimal isometries} \hfill \par

We conclude by providing an example of linearization instability. To this aim, we first need the following result due to Fischer-Marsden. For the convenience of the reader, we recall the argument here below.

\begin{lemma}[{\cite[Lemma 7.1]{fischer-marsden}}]
Fix $g_{\zero} \in \eM_{\rm H}$. If $\scal^{\rm Ch}$ is linearization stable at $g_{\zero}$, then
\begin{equation} \label{eq:stab2}
\int_M u\,(\scal^{\rm Ch})_{g_{\zero}}''(h,h)\, \nu_{g_{\zero}} = 0
\end{equation}
for any $h \in \ker(\g_{g_{\zero}})$, for any $u \in \ker(\g_{g_{\zero}}^*)$.
\end{lemma}

\begin{proof}
Fix $h \in \ker(\g_{g_{\zero}})$ and $u \in \ker(\g_{g_{\zero}}^*)$. Since $\scal^{\rm Ch}$ is linearization stable at $g_{\zero}$, there exists a smooth path $g : (-\epsilon,\epsilon) \to \eM_{\rm H}$ such that $g(0)=g_{\zero}$, $\dot{g}(0)=h$ and $\scal^{\rm Ch}(g(t))=\scal^{\rm Ch}(g_{\zero})$ for any $-\epsilon<t<\epsilon$. Differentiating this last equation twice, we get
$$
0 = \tfrac{\diff^2}{\diff t^2}\scal^{\rm Ch}(g(t))\big|_{t=0} = (\scal^{\rm Ch})_g''(h,h) + (\scal^{\rm Ch})_g'(g''(0))
$$
and so, since $\g_{g_{\zero}}^*(u)$, we get
$$\begin{aligned}
0 &= \int_M u\,(\scal^{\rm Ch})_g''(h,h)\, \nu_{g_{\zero}} + \int_M u\,(\scal^{\rm Ch})_g'(g''(0))\, \nu_{g_{\zero}} \\
&= \int_M u\,(\scal^{\rm Ch})_g''(h,h)\, \nu_{g_{\zero}} + \langle u, \g_{g_{\zero}}(g''(0)) \rangle_{L^2} \\
&= \int_M u\,(\scal^{\rm Ch})_g''(h,h)\, \nu_{g_{\zero}} \,\,
\end{aligned}$$
which concludes the proof.
\end{proof}

In the Theorem \ref{thm:c}, we consider a K\"ahler-Einstein metric with positive scalar curvature. Notice that K\"ahler-Einstein metrics with negative scalar curvature do satisfy the hypothesis of Theorem \ref{thm:a}, and so they are necessarily linearization stable.

\begin{proof}[Proof of Theorem \ref{thm:c}]
Let $\l_{\zero}\=\scal(g_{\zero})>0$ be the scalar curvature of $g_{\zero}$.
By \eqref{slin}, \eqref{eq:gamma*}, \eqref{slin2K} and by hypothesis, we have
$$\begin{gathered}
\g_{g_{\zero}}(h) = \tfrac12 \big(\D_{g_{\zero}}(\tr^{\bR}h) -\tfrac{\l_{\zero}}{m}(\tr^{\bR}h)\big) \,\, , \quad \g_{g_{\zero}}^*(u) = \tfrac12 \big(\D_{g_{\zero}}u -\tfrac{\l_{\zero}}{m}u\big)\!\Id \,\, , \\
(\scal^{\rm Ch})_{g_{\zero}}''(h,h) = g_{\zero}(\Hess_{g_{\zero}}(\tr^{\bR}h),h) +\tfrac{\l_{\zero}}{2m}|h|_{g_{\zero}}^2 \,\, .
\end{gathered}$$
It is known that the vector space of Killing vector fields is isomorphic to the space of smooth functions $f$ such that $\D_{g_{\zero}}f=\frac{\l_{\zero}}{m}f$, see e.g. \cite[p 96]{kobayashi-transf}.
Therefore, by hypothesis, $\tfrac{\l_{\zero}}{m}$ is an eigenvalue of the Laplace-Beltrami operator $\D_{g_{\zero}}$, so that $\ker(\g_{g_{\zero}}^*) \neq \{0\}$. Assume also by contradiction that $\scal^{\rm Ch}$ is linearization stable at $g_{\zero}$. Then, by means of \eqref{eq:stab2}, it follows that
\begin{equation} \label{contradictionstab}
\int_M u\,|h|_{g_{\zero}}^2\, \nu_{g_{\zero}} = 0 \quad \text{ for any } u \in \ker(\g_{g_{\zero}}^*) \,\, , \,\,\, h \in \cC^{\infty}(M,\Sym^{1,1}_0(TM)) \,\, ,
\end{equation}
where $\Sym^{1,1}_0(TM)$ denotes the subbundle of elements in $\Sym^{1,1}(TM)$ with zero trace. Fix $u \in \cC^{\infty}(M,\bR)$ such that $\D_{g_{\zero}}u =\tfrac{\l_{\zero}}{m}u$ and $u \not\equiv 0$. By means of the Stokes' Theorem, it follows that $$\int_M u\, \nu_{g_{\zero}} = 0$$ and so there exists an open ball $B \subset M$ such that $u(x)>0$ for any $x \in B$. Fix $h \in \cC^{\infty}(M,\Sym^{1,1}_0(TM))$ and, up to shrinking $B$, assume that $|h|_{g_{\zero}}^2>0$ on the whole $B$. Pick a smaller ball $B' \subset B$ and a function $\psi \in \cC^{\infty}(M,\bR)$ with ${\rm supp}(\psi) \subset B$ and such that $\psi(x) =1$ for any $x \in B'$. Then, since $\tr^{\bR}(\psi h) = 0$, by \eqref{contradictionstab} we get
$$
0 = \int_M u\,|\psi h|_{g_{\zero}}^2\, \nu_{g_{\zero}} \geq \int_{B'} u\, |h|_{g_{\zero}}^2\,\nu_{g_{\zero}} >0 \,\, ,
$$
which is not possible.
\end{proof}


\end{document}